\documentclass[11pt]{amsart}
\usepackage{amsmath,amssymb,epsfig,color, amsthm, mathrsfs}
\usepackage{graphicx}
\usepackage{multirow}
\usepackage{hhline}
\usepackage{url}
\usepackage{color}
\graphicspath{ {images/} }
\newtheorem{theorem}{Theorem}[section]

\newtheorem{corollary}[theorem]{Corollary}
\newtheorem{definition}[theorem]{Definition}
\newtheorem{lemma}[theorem]{Lemma}
\newtheorem{proposition}[theorem]{Proposition}
\newtheorem{remark}[theorem]{Remark}

\newcommand{\vanish}[1]{}\parskip=12pt

\def\p{\prime}

\def\T{\mathcal{T}}
\def\U{\mathcal{U}}
\def\V{\mathcal{V}}
\def\F{\mathcal{F}}

\numberwithin{equation}{section}

\begin{document}

\title{The Braid Index Of Reduced Alternating Links}
\author{Yuanan Diao, G\'abor Hetyei and Pengyu Liu}
\address{Department of Mathematics and Statistics, UNC Charlotte,
    Charlotte, NC 28223}
\email{}

\begin{abstract}
It is well known that the minimum crossing number of an alternating link equals the number of crossings in any reduced alternating link diagram of the link. This remarkable result is an application of the Jones polynomial. In the case of the braid index of an alternating link, Murasugi had conjectured that the number of Seifert circles in a  reduced alternating diagram of the link equals the braid index of the link. This conjecture turned out to be false. In this paper we prove the next best thing that one could hope for: we characterize exactly those alternating links for which their braid indices equal to the numbers of Seifert circles in their corresponding reduced alternating link diagrams. More specifically, we prove that if $D$ is a reduced alternating link diagram of an alternating link $L$, then $b(L)$, the braid index of $L$, equals the number of Seifert circles in $D$ if and only if $G_S(D)$ contains no edges of weight one. Here $G_S(D)$, called the Seifert graph of $D$, is an edge weighted simple graph obtained from $D$ by identifying each Seifert circle of $D$ as a vertex of $G_S(D)$ such that two vertices in $G_S(D)$ are connected by an edge if and only if the two corresponding Seifert circles share crossings between them in $D$ and that the weight of the edge is the number of crossings between the two Seifert circles. This result is partly based on the well known MFW inequality, which states that the $a$-span of the HOMFLY polynomial of $L$ is a lower bound of $2b(L)-2$, as well as a result due to Yamada, which states that the minimum number of Seifert circles over all link diagrams of $L$ equals $b(L)$. 
\end{abstract}

\maketitle
\section{Introduction}

\medskip
In 1900, Tait made a few famous conjectures in knot theory \cite{Tait}, one of which states that reduced alternating diagrams have minimal link crossing number, that is, if a link $L$ admits a projection diagram $D$ that is reduced and alternating, then the number of crossings in $D$ is the minimum number of crossings over all projections of $L$. This conjecture remained open for nearly a century until the great discovery of the Jones polynomial \cite{Jo}. An important property of the Jones polynomial is that its span (namely the difference between the highest power and the lowest power in the polynomial) is always less than or equal to the number of crossings in the knot diagram used for its calculation. Using this property, Kauffman (1987, \cite{K}), Murasugi (1987, \cite{M1}) and Thistlethwaite (1987, \cite{T}) independently proved this conjecture by establishing an equality between the span of the Jones polynomial of a reduced alternating link diagram and the number of crossings in the diagram. This is one of the greatest applications of the Jones polynomial.

This paper deals with a different link invariant, namely the {\em braid index} of an oriented link (all links in this paper are oriented links hence from now on we will only use the term links). It is well known that any link can be represented by the closure of a braid. The minimum  number of strands needed in a braid whose closure represents a given link is called the 
braid index of the link. In the case of braid index, the HOMFLY polynomial (a polynomial of two variables $z$ and $a$ that generalizes the Jones polynomial \cite{Fr, Pr}) plays a role similar to that of the Jones polynomial to the minimal crossing number. In ~\cite{Mo}, H.\ Morton showed that the
number of Seifert circles of a link $L$ is bounded from below by $a$-span$/2+1$ (which is called the Morton-Frank-Williams inequality, or MFW inequality
for short). Combining this with a result due to Yamada which states the braid index of an oriented link $L$ equals the minimum number of Seifert circles of all link
diagrams of $L$~\cite{Ya}, it follows that the braid index of $L$ is bounded from below by $a$-span$/2+1$. In analogy to the crossing number conjecture for a reduced alternating link diagram, it would be natural for one to hope that equality in the MFW inequality would hold for reduced alternating links. This was indeed conjectured by Murasugi in \cite{Mu}. This conjecture turned out to be false and counter examples are plentiful. The simplest knot that serves as a counter example is $5_2$: its minimum diagram has 4 Seifert circles but its $a$-span is only 4 (so $a$-span$/2+1=3<4$). In fact the braid index of $5_2$ is also 3, not 4.
Subsequent research then focused on identifying specific link classes for which the equality in the MFW inequality holds. Examples include the closed positive braids with a full twist (in particular the torus links)~\cite{FW}, the 2-bridge links and fibered alternating links~\cite{Mu}, and a new class of links discussed in a more recent paper \cite{Lee}. For more readings on related topics, interested readers can refer to \cite{Bir, Crom, El, MS, Na1, Sto}.

The main result of this paper is the complete characterization of those reduced alternating links for which the numbers of Seifert circles in their corresponding link diagrams equal to their braid indices. More specifically, let $D$ be a reduced alternating link diagram of a link $L$, define a simple and edge weighted graph, called the {\em Seifert graph} and denoted by $G_S(D)$, as follows. Every Seifert circle corresponds to a vertex in $G_S(D)$. Two vertices in $G_S(D)$ are connected by an edge if (and only if) the two corresponding Seifert circles in $D$ share crossings. The weight of an edge in $G_S(D)$ is the number of crossings between the two Seifert circles corresponding to the two vertices in $G_S(D)$. Our main result is stated in terms of $G_S(D)$ in a very simple way:

\begin{theorem}\label{main}
Let $L$ be a link with a reduced alternating diagram $D$, then the braid index of $L$ equals the number of Seifert circles in $D$ if and only if $G_S(D)$ is free of edges of weight one.
\end{theorem}

The proof of the theorem contains two parts. In the easier part, we show that if the link diagram $D$ of a link $L$ contains an edge of weight one, then the braid index of $L$ is less than the number of Seifert circles in $D$. In this proof $D$ does not have to be reduced nor alternating. In the harder part of the proof, we show that if $D$ is reduced, alternating and $G_S(D)$ is free of edges of weight one, then the equality in the MFW inequality holds hence the number of Seifert circles in $D$ equals the braid index of $L$.

This paper is structured as follows. In Section~\ref{s2}, we introduce the HOMFLY
polynomial and the resolving trees. In
Section~\ref{s3}, we introduce an important new concept called the {\em intermediate Seifert circles} (IS circles for short) and the decomposition of a link diagram into the IS circles. In Section~\ref{s4}, we introduce another important concept called the {\em castle}. This is actually a structure that exists within any link diagram that is similar to a braid locally. Based on the castle structure, we then develop two algorithms that will provide us two special resolving trees which will play a crucial role in proving our main theorem. The proof of our main theorem is given in Section~\ref{s5}. We end the paper with some remarks and observations in Section~\ref{s6}.

\section{HOMFLY polynomial and resolving trees}\label{s2}

For the sake of convenience, from this point on, when we talk about a
link diagram $D$, it is with the understanding that it is the link
diagram of some link $L$. Since we will be talking about the link
invariant such as braid index and the HOMFLY polynomial, it is safe for
us to use $D$ as a link without mentioning $L$. Let $D_+$, $D_-$, and
$D_0$ be oriented link diagrams of a link $L$ that coincide except at a
small neighborhood of a crossing where the diagrams are presented as in
Figure~\ref{fig:cross}. We say the crossing presented in $D_+$ has a
{\em positive} sign and the crossing presented in $D_-$ has a {\em negative} sign. The following result appears in~\cite{Fr,Ja}.

\begin{proposition}\label{Ho}
There is a unique function that maps each oriented link diagram $D$ to a two-variable Laurent polynomial with integer coefficients $P(D,z,a)$ such that 
\begin{enumerate}
\item If $D_1$ and $D_2$ are ambient isotopic, then $P(D_1,z,a)=P(D_2,z,a)$.
\item $aP(D_+,z,a) - a^{-1}P(D_-,z,a) = zP(D_0,z,a)$. 
\item If $D$ is an unknot, then $P(D,z,a)=1$. 
\end{enumerate}
\end{proposition}

\begin{figure}[htb!]
\includegraphics[scale=.6]{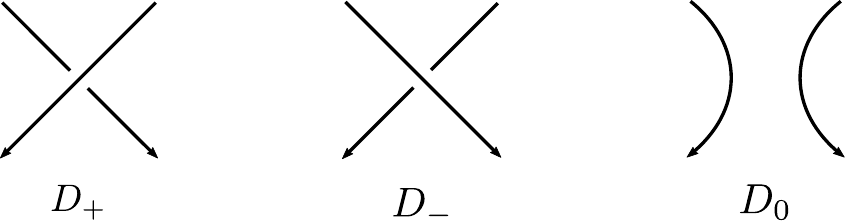}
\caption{The sign convention at a crossing of an oriented link and the splitting of the crossing: the crossing in $D_+$ ($D_-$) is positive (negative) and is assigned $+1$ ($-1$) in the calculation of the writhe of the link diagram.}
\label{fig:cross}
\end{figure}

The Laurent polynomial $P(D,z,a)$ is called the {\em HOMFLY polynomial}
of the oriented link $D$. The second condition in the proposition is
called the {\em skein relation} of the HOMFLY polynomial. With conditions (2) and (3) above, one can easily show that if  $D$ is a trivial link with $n$ connected components, then $P(D,z,a)=((a-a^{-1})z^{-1})^{n-1}$ (by applying these two conditions repeatedly to a simple closed curve with $n-1$ twists in its projection).
For our purposes, we will actually be using the following two equivalent forms of the skein relation:
\begin{eqnarray}
P(D_+,z,a)&=&a^{-2}P(D_-,z,a)+a^{-1}zP(D_0,z,a),\label{Skein1}\\
P(D_-,z,a)&=&a^2 P(D_+,z,a)-azP(D_0,z,a).\label{Skein2}
\end{eqnarray}

A rooted and edge-weighted binary tree $\T$ is called a {\em resolving tree} of an oriented link diagram $D$ (for the HOMFLY polynomial)
if the following conditions hold. First, every vertex of $\T$ corresponds to an oriented link diagram. Second, the root vertex of $\T$ corresponds to the original link diagram $D$. Third, each leaf vertex of $\T$ corresponds to a trivial link. Fourth, if we direct $\T$ using the directions of the paths from the root vertex to the leaf vertices, then under this direction any internal vertex has exactly two children vertices and the corresponding link diagrams of these three vertices are identical except at one crossing and they are related by one of the two possible relations at that crossing as shown in Figure~\ref{fig:rtree}, where the edges are weighted and the directions of the edges coincide with the direction of $\T$.

\begin{figure}[htb!]
\includegraphics[scale=.35]{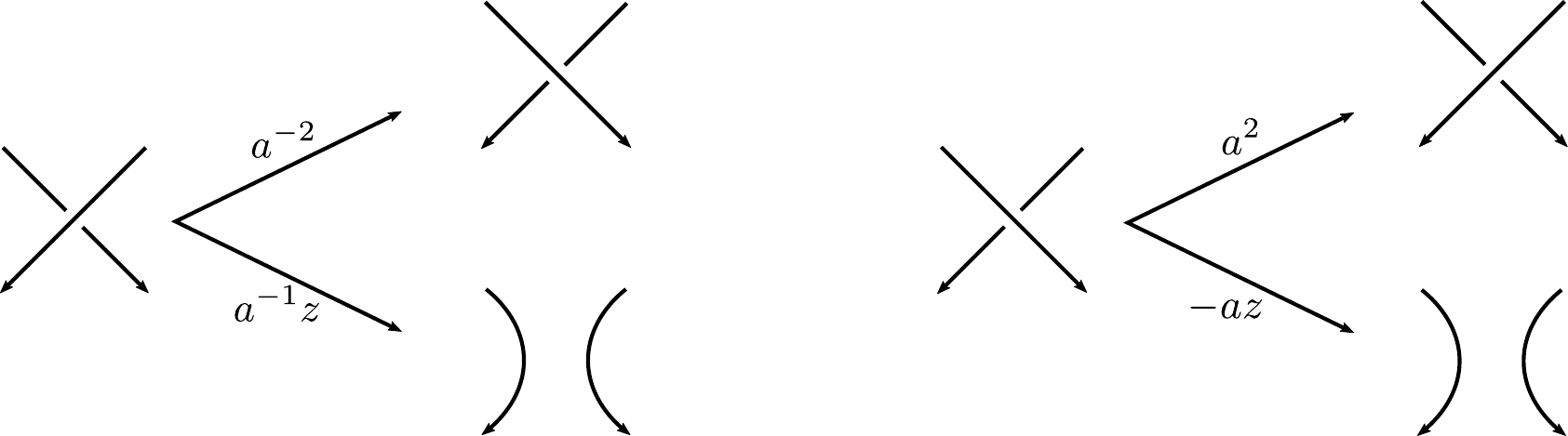}
\caption{A pictorial view of how the edge weights are assigned by the skein relations (\ref{Skein1}) and (\ref{Skein2}) to the edges connecting an internal vertex to the two vertices that it precedes in the resolving tree.}
\label{fig:rtree}
\end{figure}

If $D$ admits a resolving tree $\T$, then one can easily show that $P(D,z,a)$ is a summation in which each leaf vertex of $\T$ contributes exactly one term in the following way. Let $\U$ be the trivial link corresponding to a leaf vertex in $\T$ and let $Q$ be the unique path from the root $D$ to the leaf vertex $\U$. Then the contribution of the leaf vertex is simply $((a-a^{-1})z^{-1})^{\gamma(\U)-1}$ multiplied by the weights of the edges in $Q$, where $\gamma(\U)$ is the number of components in $\U$. Let  $w(\U)$ be the writhe of $\U$, $\gamma(\U)$ be the number of components in ${\U}$, $t(\U)$ be the number of smoothed crossings (in $D$ in order to obtain $\U$) and $t^{-}(\U)$ be the number of smoothed crossings that are negative. As shown in Figure~\ref{fig:rtree}, the degree of $a$ in the weight of an edge is exactly the change of writhe from the starting vertex of the edge (remember that it is directed from the root to the leaf) to the ending vertex of the edge, whereas a $z$ term in the weight of the edge indicates that the ending vertex is obtained from the starting vertex by a crossing smoothing and a negative sign in the weight indicates that the smoothed crossing is a negative crossing. It follows that the total contribution of  $\U$ in $P(D,z,a)$ is

\begin{eqnarray}
&&(-1)^{t^{-}(\U)}z^{t(\U)}a^{w(\U)-w({D})}((a-a^{-1})z^{-1})^{\gamma(\U)-1},
\end{eqnarray}
hence 
\begin{eqnarray}
P(D,z,a)&=& \sum_{\U\in \T}(-1)^{t^{-}(\U)}z^{t(\U)}a^{w(\U)-w({D})}((a-a^{-1})z^{-1})^{\gamma(\U)-1}.
\end{eqnarray}

It follows that the highest and lowest $a$-power terms that $\U$ contributes to $P(D,z,a)$ are
\begin{eqnarray}\label{apower1}
&&(-1)^{t^{-}(\U)}z^{t(\U)-\gamma(\U)+1}a^{w(\U)-w({D})+\gamma(\U)-1}
\end{eqnarray}
and 
\begin{eqnarray}\label{apower2}
&&(-1)^{t^{-}(\U)+\gamma(\U)-1}z^{t(\U)-\gamma(\U)+1}a^{w(\U)-w({D})-\gamma(\U)+1}
\end{eqnarray}
respectively.

It is well known that resolving trees exist for any given oriented link diagram $D$. We will describe two algorithms for constructing resolving trees with some special properties that we need at the end of Section \ref{s4}.

\section{Seifert graphs and intermediate Seifert circle decompositions}\label{s3}

For the purpose of this section, there is no need for us to specify the over and under pass in a link diagram $D$. Thus in this section $D$ is treated as a collection of oriented plane closed curves that may intersect each other and may have self intersections, with the intersection points being the crossings in the link diagram. So we will be using the term ``crossing" and ``intersection" interchangeably.

\begin{definition}{\em
Let $D$ be a link diagram and $S$ be its Seifert circle decomposition. We construct a graph $G_S(D)$ from $S$ by identifying each Seifert circle in $S$ as a vertex of $G_S(D)$. If there exist $k\ge 1$ crossings between two Seifert circles $C_1$ and $C_2$ in $S$, then the two corresponding vertices (also named by $C_1$ and $C_2$) are connected by an edge with weight $k$. Otherwise there will be no edges between the two vertices. The edge-weighted (simple) graph $G_S(D)$ is called the {\em Seifert graph} of $D$.}
\end{definition}

\begin{definition}{\em
Let $D$ be a link diagram. A simple closed curve obtained from $D$ by smoothing some vertices in $D$ is called an {\em intermediate Seifert circle} (which we will call an IS circle for short). A decomposition of $D$ into a collection of IS circles is called an {\em IS decomposition} of $D$ and is denoted by $I(D)$. The number of IS circles in $I(D)$ is denoted by $s(I(D))$. In particular, if $I(D)$ is the Seifert circle decomposition of $D$, then $s(I(D))=s(D)$ is the number of Seifert circles in $D$.}
\end{definition}

Note: The IS circles may be linked to each other, but this is not important for the purpose of this section and we will treat them as plane curves that may intersect each other.

\begin{definition}\label{loop}{\em 
Let $D$ be a link diagram. For a given point $p$ on $D$ that
is not on a crossing, consider the component of $D$ that contains $p$.  Let us travel along this component starting from $p$ following the orientation of the component. As we travel we ignore the crossings that we encounter the first time. Eventually we will arrive at the first crossing that we have already visited (which is in fact the first crossing involving strands of this component). We call this crossing a {\em loop crossing}, since smoothing it results in two curves: the part that we had traveled between the two visits to this loop crossing (which is an IS circle since it does not contain crossings in itself, and it does not contain $p$), and the other part that still contains $p$, which is now a new component in the new link diagram obtained after smoothing the loop crossing.  For this new link component that contains $p$, we will start from $p$ and continue this process and obtain new IS circles. This process ends when the new link component containing $p$ is itself an IS circle hence traveling along it will not create any new loop crossings. Thus, by smoothing the loop crossings defined (encountered) this way, we can decompose the component of $D$ that contains $p$ into a collection of IS circles. By choosing a point on each component of $D$ and repeat the above process, we obtain an IS circle decomposition of $D$. }
\end{definition}

\begin{figure}[htb!]
\includegraphics[scale=.3]{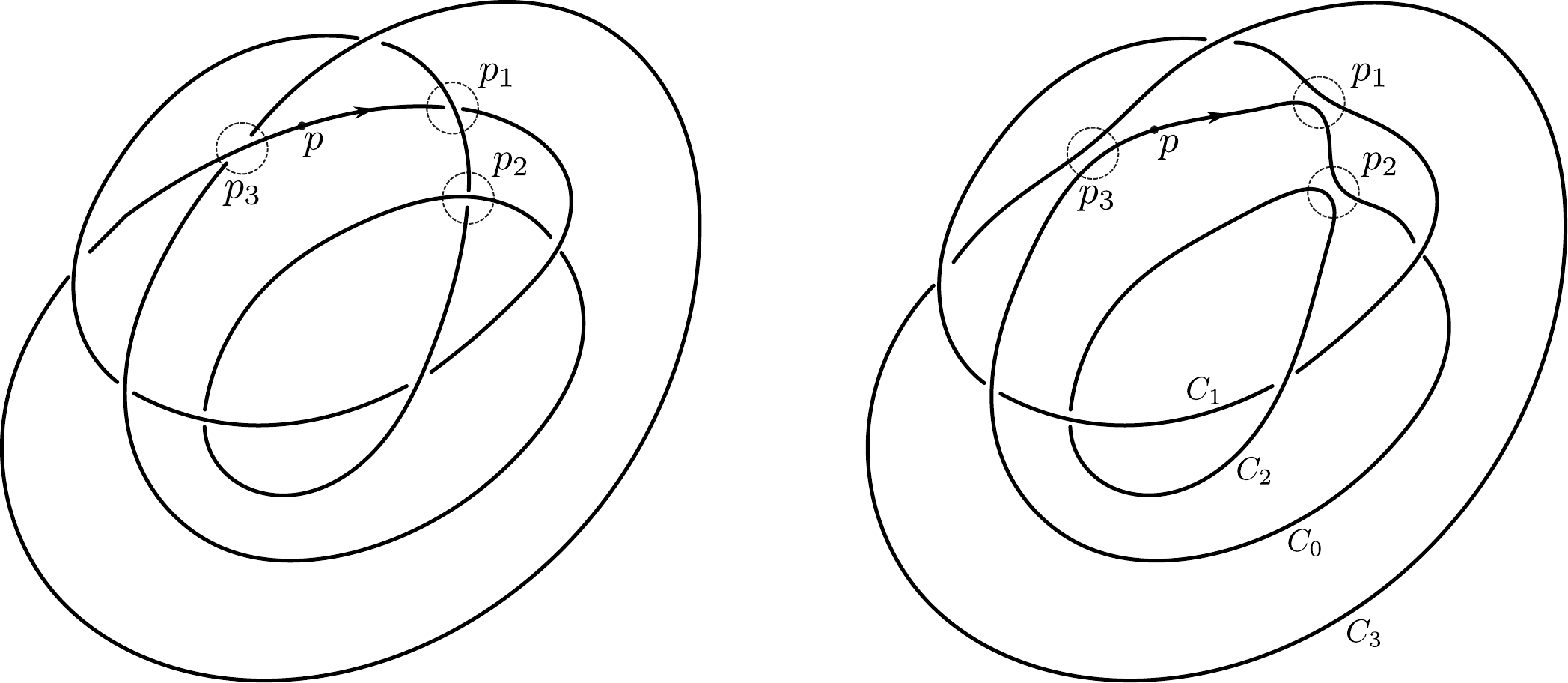}
\caption{The dot marks the starting point $p$ on $D$. $p_1$, $p_2$ and $p_3$ are the loop crossings obtained in that order. In this case $D$ is decomposed into four IS circles $C_0$, $C_1$, $C_2$ and $C_3$ as shown.}
\label{fig:loop}
\end{figure}

\begin{remark}\label{re_loop}{\em 
The loop crossings (hence the corresponding IS circles obtained by
smoothing them) on a link component are uniquely determined by $p$. See
Figure \ref{fig:loop} for an example. However, a different choice of $p$ may result in different loop crossings and different IS circle decompositions of $D$, and certainly there are other ways to obtain different IS circle decompositions of $D$.}
\end{remark}

\begin{remark}\label{re_loop2}{\em 
Let $I(D)$ be an IS circle decomposition of $D$. If $I(D)$ contains crossings (so it is not the Seifert circle decomposition of $D$) then we can obtain a new IS circle decomposition $I^\p(D)$ with less crossings from $I(D)$ by the following operation. Let $\tilde{C}_1$ and $\tilde{C}_2$ be two IS circles in $I(D)$ that intersect each other.  They must intersect each other an even number of times. Consider two  crossings between $\tilde{C}_1$ and $\tilde{C}_2$ that are consecutive as we travel along either $\tilde{C}_1$ or $\tilde{C}_2$. Smoothing these crossings results in two closed curves with two possible cases: (1) These closed curves are themselves IS circles. In this case, replacing $\tilde{C}_1$ and $\tilde{C}_2$ in $I(D)$ by these two new IS circles results in the new IS circle decomposition $I^\p(D)$ with $s(I(D))=s(I^\p(D))$. If this is the case we say that $I(D)$ is reducible and the new IS circle decomposition is said to be obtained from $I(D)$ by a {\em reduction operation}. (2) These closed curves contain self intersection points hence loop crossings. If we smooth the loop crossings, then we obtain at least 3 IS circles. Thus replacing $\tilde{C}_1$ and $\tilde{C}_2$ in $I(D)$ by these new IS circles results in the new IS circle decomposition $I^\p(D)$ with $s(I(D))<s(I^\p(D))$.  Since each operation leads to a new IS circle decomposition without decreasing the number of IS circles and smoothes at least two crossings, the following lemma is immediate. 
}
\end{remark}

\begin{lemma}\label{new_lemma2}
Let $I_1(D)$ be any IS decomposition of $D$ and $I_2(D)$ be an IS decomposition of $D$ obtained from $I_1(D)$ by a sequence of operations as described in Remark \ref{re_loop2},  then the number of IS circles in $I_1(D)$ is less than or equal to the number of IS circles in $I_2(D)$. In particular, the number of IS circles in $I_1(D)$ is less than or equal to the number of Seifert circles in $D$.
\end{lemma}

From this point on, we will be using figures that contain Seifert circles to provide examples and to explain our ideas in the proofs. To avoid possible confusion, let us point out how to read our drawings. Although we always draw the Seifert circles as closed curves (mainly for the sake of convenience), when we reconstruct the original link diagram from them, it is important to remember that parts of the Seifert circles are not on the diagram as shown in Figure \ref{Figure_extra}.

\begin{figure}[htb!]
\includegraphics[scale=.4]{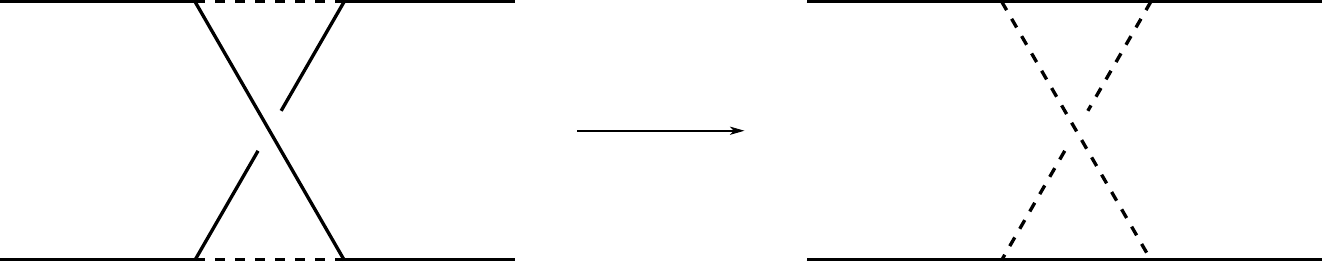}
\caption{Left: The link diagram with the crossing being part of it; Right: The Seifert circles after the crossing is smoothed.}
\label{Figure_extra}
\end{figure}

In an IS decomposition $I(D)$ of $D$, notice that if an IS circle $\tilde{C}$ has no crossings with any other IS circles, then it is itself a Seifert circle. If $\tilde{C}$ has crossings with other IS circles, then these crossings divide $\tilde{C}$ into arcs. As we travel along $\tilde{C}$, we travel along these arcs (called the {\em dividing arcs of $\tilde{C}$}), say $\tau_1$, $\tau_2$, ..., $\tau_k$ (and $\tau_k$ connects back to $\tau_1$) in that order and $\tau_j$ belongs to Seifert circle $C_j$ ($1\le j\le k$), and obtain a directed and closed walk $C_1C_2\cdots C_kC_1$ in $G_S(D)$. We call this closed walk the {\em Seifert circle walk} of $\tilde{C}$. Let $C_1^{\p}C_2^{\p}\cdots C_m^{\p}C_1^{\p}$ be a cycle in $G_S(D)$, {\em i.e.}, $C_i^{\p}\not=C_j^{\p}$ if $i\not=j$ and $C_j^{\p}$ shares crossings with $C_{j+1}^{\p}$ ($C_{m+1}^{\p}=C_1^{\p}$) in $D$. It is necessary that $m$ is even and $m\ge 4$. Furthermore, Seifert circles in a cycle cannot cannot be concentric to each other with only one possible exception in which one Seifert circle in the cycle contains all other Seifert circles. Because of this, a cycle in $G_S(D)$ bounds a region in the plane if we consider $G_S(D)$ as a plane graph. A cycle $G_S(D)$ is said to be {\em inner most} if there are no other cycles in the region that it bounds. We say that $C_1C_2\cdots C_kC_1$ contains a cycle of $G_S(D)$ if there exists a cycle $C_1^{\p}C_2^{\p}\cdots C_m^{\p}C_1^{\p}$ in $G_S(D)$ such that $\tilde{C}$ passes through $C_{j}^{\p}C_{j+1}^{\p}$ ($1\le j\le m$) in that order.

\begin{lemma}\label{lemma4}
If an IS decomposition $I(D)$ of $D$ contains an IS circle whose Seifert circle walk contains a cycle of $G_S(D)$, then the number of IS circles in $I(D)$ is strictly less than $n$, the number of Seifert circles in $D$. 
\end{lemma}

Proof. Let $\tilde{C}$ be an IS circle in $I(D)$ whose Seifert circle walk contains a cycle of $G_S(D)$. Let $D^\p$ be the diagram corresponding to $I(D)$ with $\tilde{C}$ removed. The set of IS circles $I(D)-\{\tilde{C}\}$ forms an IS circle decomposition of $D^\p$. Consider the diagram $D_1$ obtained from $D$ by smoothing all crossings not on $\tilde{C}$. This results in a new IS circle decomposition which we will denote by $I_{\tilde{C}}(D)$. The diagram $D_1$ can also be obtained by smoothing all crossings in $D^{\p}$ first and then adding $\tilde{C}$ back to it. Thus every IS circle in $D_1$ is a Seifert circle of $D^\p$ except $\tilde{C}$. By Lemma \ref{new_lemma2}, $s(I(D))-1=s(I(D)-\{\tilde{C}\})\le s(D^\p)=s(I_{\tilde{C}}(D))-1$, thus $s(I(D))\le s(I_{\tilde{C}}(D))\le s(D)$. Notice that smoothing crossings not on $\tilde{C}$ does not change the Seifert circle walk of $\tilde{C}$, and all crossings in $D_1$ are passed by $\tilde{C}$. 
We make the following additional observations.

Observation 1. If a consecutive sub-walk $C_1C_2\cdots C_mC_1$ in the Seifert circle walk of $\tilde{C}$ is a cycle in $G_S(D_1)$ ($m\ge 4$ and is even), then $C_2$ and $C_3$ can only share one crossing in $D_1$. This can be seen by traveling along the portion $\tau$ of $\tilde{C}$ corresponding to this sub-walk, see Figure \ref{cycle_limit}. The arcs $C_2-\tau$ and $C_3-\tau$ are in the two separate regions bounded by the simple closed curve defined by the union of $\tau$ and the arc $qp$ of $C_1$ as shown in Figure \ref{cycle_limit}, hence there can be no crossings between them by the Jordan close curve theorem. Notice that this observation is not affected if one of these Seifert circles contains the rest in its interior (in which case we may reroute a segment on one of the Seifert circles in an obvious manner).

\begin{figure}[htb!]
\includegraphics[scale=.45]{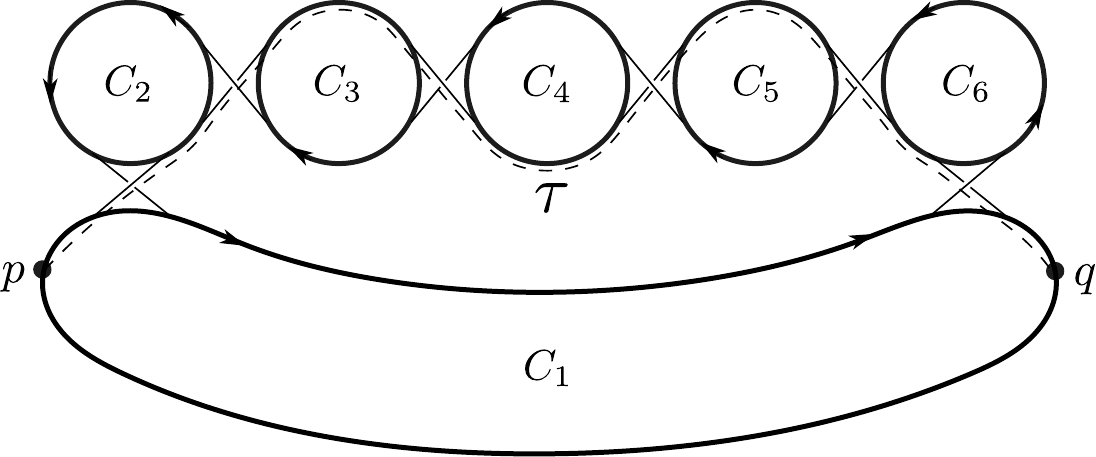}
\caption{The union of the dividing arcs ($\tau$) running through $C_1C_2\cdots C_mC_1$ ($m=6$ is shown here) is marked by the dashed curve.
}
\label{cycle_limit}
\end{figure}

Observation 2. Any operation on the Seifert circle walk of $\tilde{C}$ as defined in Remark \ref{re_loop2} is a reduction operation. Recall that such a reduction operation is performed at two consecutive crossings on $\tilde{C}$ corresponding to a sub-walk of the form $C_1^{\p}C_2^{\p}C_1^{\p}$, and it results in a new long IS circle $\tilde{C}^\p$ whose Seifert circle is obtained from that of $\tilde{C}$ by reducing the sub-walk $C_1^{\p}C_2^{\p}C_1^{\p}$ to $C_1^{\p}$. Since a closed sub-walk that does not contain a cycle contains at least a sub-walk of the form $C_1^{\p}C_2^{\p}C_1^{\p}$, and the corresponding sub-walk after the reduction operation results in a closed sub-walk that still does not contain a cycle, this reduction can be repeated until the sub-walk is reduced to a single Seifert circle.

Assume that the Seifert circle walk $C_1C_2\cdots C_kC_1$ of $\tilde{C}$ contains a cycle and consider a shortest closed sub-walk on it that contains a cycle of $G_S(D)$. By definition, any proper closed sub-walk of this sub-walk is cycle free and its underlying set of edges is a tree containing at least one leaf where we must have a reducible sub-walk of the form $C_1^{\p}C_2^{\p}C_1^{\p}$ by Observation 2 above, these sub-walks do not contain cycles hence each can be reduced to 
a single Seifert circle. The Seifert circle walk of the resulting long
IS circle $\tilde{C}^\p$ contains a sub-walk of the form  $C_1^{\p}C_2^{\p}\cdots C_m^{\p}C_1^{\p}$ where $C_1^{\p}C_2^{\p}\cdots C_m^{\p}C_1^{\p}$ is a cycle in $G_S(D)$. By Observation 1, this sub-walk cannot be reduced to a single Seifert circle as the resulting Seifert circle walk does not contain a sub-walk of the form $C_3^{\p}C_2^{\p}$ since the long IS circle only pass between $C_2^{\p}$ and $C_3^{\p}$ once. Thus the reduction operations cannot reduce the Seifert circle walk $C_1C_2\cdots C_kC_1$ to a single Seifert circle. That is, at some point, the operation defined in Remark  \ref{re_loop2} can no longer be a reduction operation and will increase the number of IS circles in the resulting IS circle decomposition, contradicting to the given condition that $s(I(D))=s(D)$. \qed

\section{Castles, floors and resolving trees}\label{s4}

In this section, we describe a structure called a {\em castle}, which is like a braid  in a local sense, and show that such structure exists in any link diagram. We then use this structure to develop two algorithms that will produce two special resolving trees for any given link diagram. The castle structure as well as these resolving trees will provide us the necessary tools in proving our main result of this paper.
 
\begin{definition} {\em Let $D$ be a link diagram. Consider three
    Seifert circles $C_1$, $C_2$ and $C_3$ in $G_S(D)$ as shown in
    Figure \ref{trap}. Notice that $C_3$ is bounded within the
    topological disk created by arcs of $C_1$, $C_2$ and the two
    crossings as shown in the figure, and that $C_3$ is connected to
    $C_2$, but not to $C_1$ (it cannot be connected to $C_1$ due to
    orientation restriction). We say that $C_3$ is {\em trapped} by $C_1$ and $C_2$. Similarly, $C_4$ is trapped by $C_2$ and $C_3$ as shown in Figure  \ref{trap}.}
\end{definition}

\begin{figure}[htb!]
\includegraphics[scale=.35]{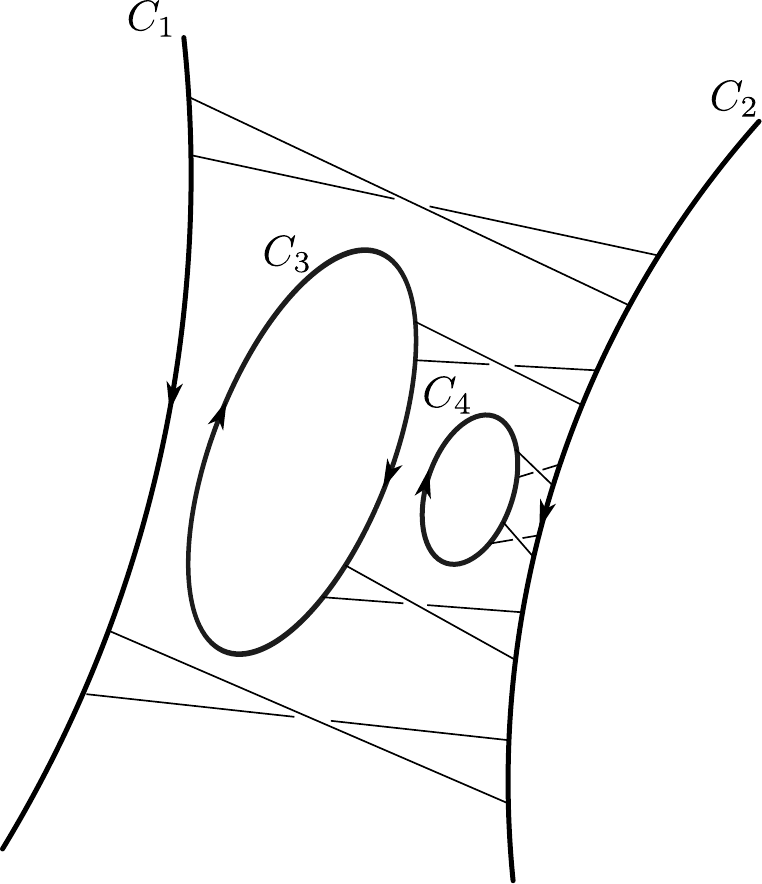}
\caption{$C_3$ is trapped by $C_1$ and $C_2$ and $C_4$ is trapped by $C_2$ and $C_3$.}
\label{trap}
\end{figure}

It is apparent from the definition that if $C_1$ traps $C_3$, then $C_3$ cannot trap $C_1$, in fact $C_3$, or any Seifert circle bounded within $C_3$, cannot trap any Seifert circle outside the disk bounded by $C_3$, $C_2$ and the two crossings as shown in Figure \ref{trap}.

\begin{definition}\label{concentric}{\em 
Seifert circles as illustrated in the left of Figure \ref{fig_concentric} are said to form a {\em concentric chain}, Seifert circles as illustrated in the right of Figure \ref{concentric} are said to form a {\em general concentric chain}. Notice that between two concentric Seifert circles we may have other Seifert circles as shown in the right of Figure \ref{fig_concentric}.}
\end{definition}

\begin{figure}[htb!]
\includegraphics[scale=.3]{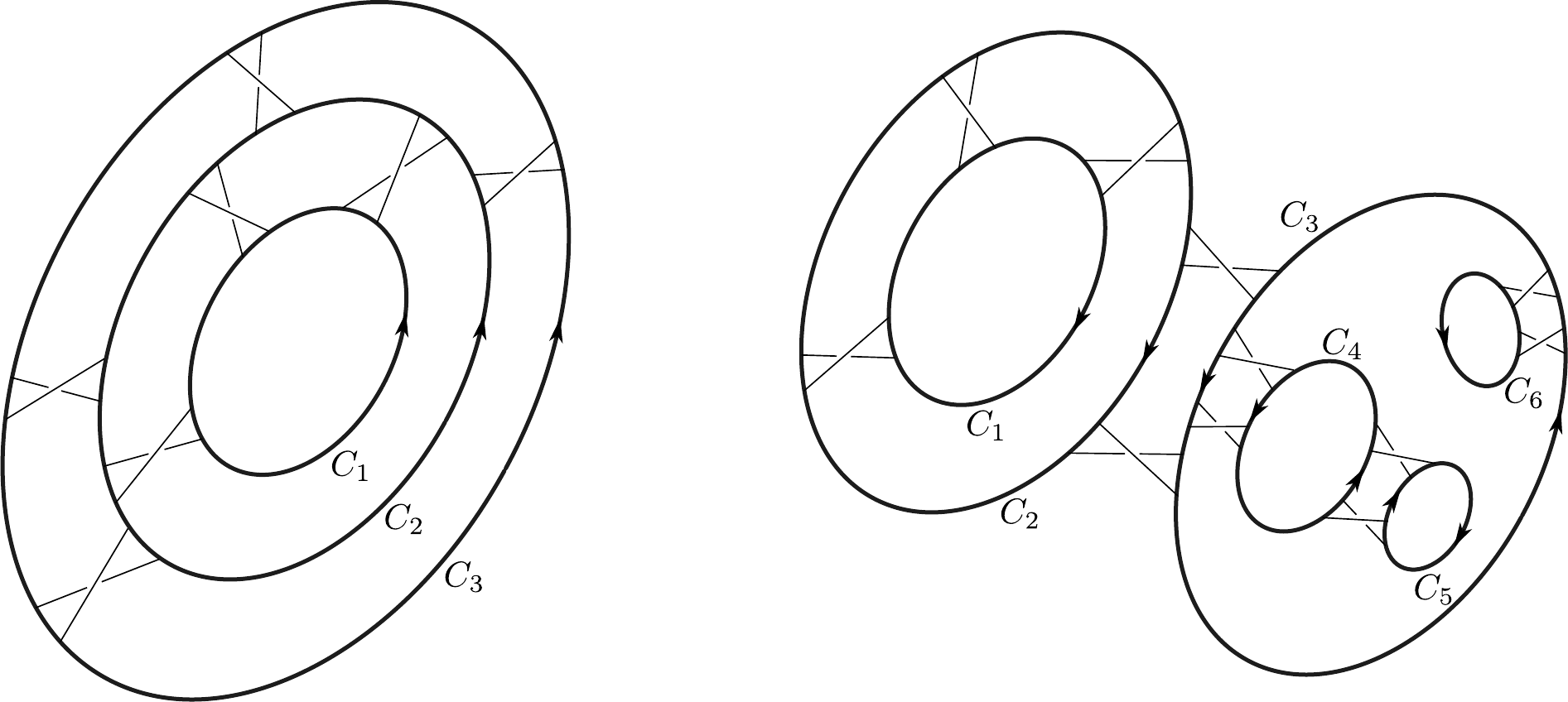}
\caption{Left: $C_1$, $C_2$ and $C_3$ form a concentric chain of Seifert circles; Right: $C_1$, $C_2$, $C_3$, $C_4$ and $C_1$, $C_2$, $C_3$, $C_6$  form two general concentric chains of Seifert circles.}
\label{fig_concentric}
\end{figure}

Let $D$ be a link diagram. Consider a Seifert circle $C$ of $D$ that is inner-most, that is, $C$ does not bound any other Seifert circles inside it. Choose a starting point on $C$ away from the places where crossings are placed on $C$. Following the orientation of $C$, we are able to order the crossings along $C$ as shown in Figure \ref{castle}, where a clockwise orientation is illustrated and for the purpose of illustration the part of $C$ from the first to the last crossing is drawn in a horizontal manner bounded between the starting point $p_0$ and ending point $q_0$ marked on it. We call this segment of $C$ the ground (0-th) floor (of the {\em castle} that is to be defined next). We now describe the procedure to build a structure on top of the ground floor that we call a {\em castle}. If $C$ is connected to another Seifert circle $C_1$ in $G_S(D)$, then there exists crossings between $C$ and $C_1$ between $p_0$ and $q_0$ and they can be ordered by the orientation of $C_1$ (which is the same as that of $C$). Let $p_{1}$ and $q_1$ be two points immediately before the first crossing and after the last crossing (so no other crossings are between $p_1$ and the first crossing, and between $q_1$ and the last crossing). The segment of $C_1$ between $p_1$ and $q_1$ is then called a {\em floor of level 1}. If $C_1$ has no crossings with other Seifert circles on this floor, then this floor terminates (for example if there is only one crossing between $C$ and $C_1$ then this floor will terminate). If $C_1$ shares crossings with another Seifert circle $C_2$ on this floor (which must be between $p_1$ and $q_1$), then we can define a second floor in a similar manner. We call the crossings between two floors {\em ladders} as we can only go up or down from a floor to the next through these crossings. This process is repeated until we reach a floor that terminates, meaning there are either no other floors on top or there are no ladders to reach the floor above. The castle is the structure that contains all possible floors (and ladders between them) constructed this way. Notice that there may be more than one separate floor on top of any given floor. If $F_k$ is a top floor, $F_{k-1}$ the floor below it, $F_{k-2}$ the floor below $F_{k-1}$, and so on, then the collection of floors $F_0$, $F_1$, ..., $F_k$, including all crossings among them, is called a {\em tower}. Notice that the Seifert circles corresponding to the floors in a tower form a general concentric Seifert circle chain hence the height of tower (the number of floors in it) is bounded above by the number of Seifert circles in $D$. Finally, between two adjacent floors we may have trapped Seifert circles that may or may not be part of the castle as shown in Figure \ref{castle}. However if there exist other floors between two adjacent floors,  the Seifert circle corresponding to the top floor will not share any crossings with the Seifert circles with floors between these two floors due to their opposite orientations. 

\begin{figure}[htb!]
\includegraphics[scale=.4]{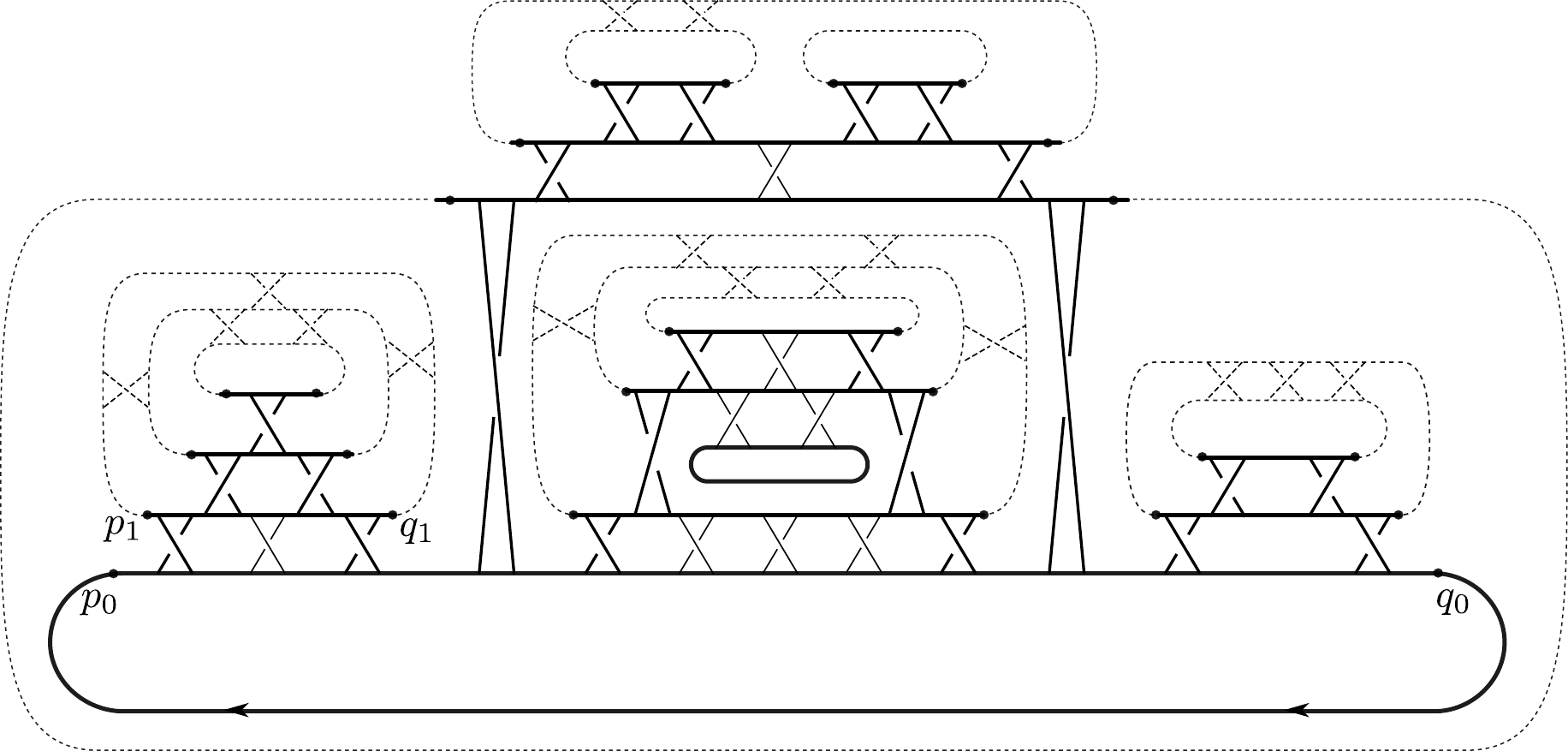}
\caption{A castle built on top of an inner most Seifert circle $C$.}
\label{castle}
\end{figure}

\begin{lemma}\label{lemma_trapfree}
For any link diagram $D$, there exists an inner most Seifert circle $C$ such that a castle built on it contains no trapped Seifert circles.
\end{lemma}

Proof. Start with any inner most Seifert circle $C_0$ and build a castle on it. If it contains no trapped Seifert circles, we are done. If not, let $C^\p$ be a Seifert circle trapped between floors $F_i$ and $F_{i+1}$ (with $C_i$ and $C_{i+1}$ being their corresponding Seifert circles). Choose an inner most Seifert circle that is contained inside $C^\p$ (use $C^\p$ itself if it does not contain any Seifert circle within it) and build a castle on it. This castle is bounded away from either $F_i$ or $F_{i+1}$ depending on the orientation of the new base Seifert circle. Since the base Seifert circle is contained in $C^\p$, if the new castle is not contained within $C^\p$ entirely, then $C^\p$ will contribute a floor to the new castle. In fact, the curves in the new castle can only exit the region between $F_i$ and $F_{i+1}$ that traps $C^\p$ through either $F_i$ (if $C^\p$ shares crossings with $F_i$) or $F_{i+1}$ (if $C^\p$ shares crossings with $F_{i+1}$), but not both. It follows that the new castle is completed contained within the towers that contain $F_i$ and $F_{i+1}$. If the castle built on this new Seifert circle again contains trapped Seifert circles, we will repeat this process. Since this process starts with new trapped Seifert circles that are bounded within the previous castles, the process will end after finitely many steps and we reach a castle without trapped Seifert circles. \qed

We now define two algorithms used to derive two special resolving trees for link diagrams. 
These resolving trees will play a key role in proving our main result in the next section. 
First let us describe the general approach to obtain a resolving tree.

We think of a resolving tree as graph of a branching process: at each internal
vertex we take a crossing of the current link diagram, and branch on
smoothing or flipping the crossing. We are growing our resolving tree by
adding two children at a time to a vertex that was a leaf up until that
point. We will divide this process into several {\em phases}. 

When we flip a crossing, the corresponding child has essentially the
same connected components as the parent. Smoothing may merge two
existing components or split one component into two. In order to
maintain some control over this process, we select a starting point,
start traversing a connected component from there, and consider
branching on the crossings as we encounter them in this traversing
process. The first phase ends when all current leaf vertices have a
connected component containing our starting point and all crossings
(smoothed, kept, or flipped) of the original link diagram along the
connected component have been visited at least once. Note that we visit
a crossing twice exactly when it is an unchanged or flipped {\em loop
  crossing}, that is, a crossing whose both strands end up in the same
connected component at the end of phase.    

A crossing is called {\em descending} if during this process we travel
along the overpassing strand first, otherwise it is {\em ascending}.    
A {\em descending operation} keeps a descending crossing unchanged (no
branching happens when we encounter the crossing) and branches on
flipping or smoothing an ascending crossing. An {\em ascending
  operation} does just the opposite: it keeps an ascending crossing
currently visited, and branches on flipping or smoothing a descending
crossing. In the first phase we will use either the ascending or the
descending operation at all crossings. Hence in all leaf nodes at the
end of phase one we have a connected component containing our starting
point, along which all crossings of the original link diagram have been
visited, and the ones that were not smoothed are now either all
descending or all ascending. Therefore the connected component
containing our starting point is either below or above the other
components, it is not linked to them. At this point we remove this
component from consideration, we proceed as if it was not present any
more. We select a new starting point, and we perform a next phase of
adding new vertices to the resolving tree, using only ascending or only
descending operations. We continue adding new phases until there is no
crossings to be considered left. The leaf vertices of the final tree will
be unlinked unknots. 

Now, we will define two specific algorithms by choosing the starting point at the beginning of each phase and assigning the appropriate descending/ascending operation for that phase.

Algorithm P: The starting point at the beginning of each phase is chosen to be the starting point of the ground floor of a castle free of trapped Seifert circles. If the Seifert circle providing the ground floor is clockwise, the descending operation is applied, otherwise the ascending operation is applied throughout the entire phase.

Algorithm N: Exchange the descending and ascending operations in Algorithm P.

\vanish{\begin{figure}[htb!]
\includegraphics[scale=.4]{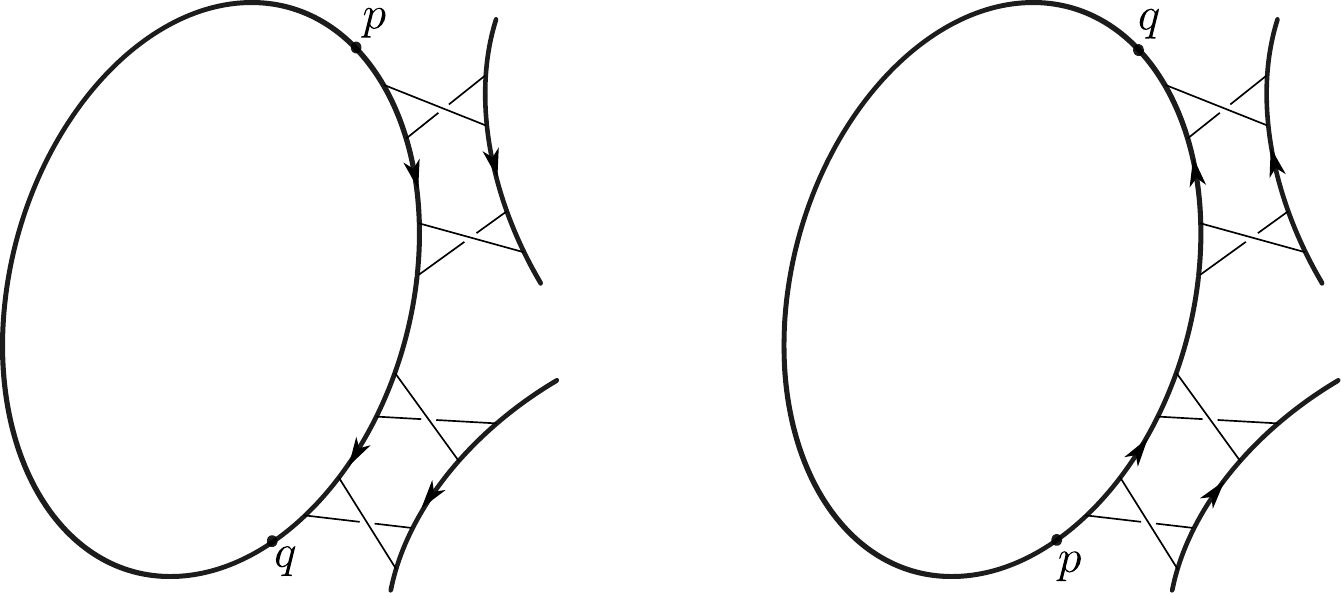}
\caption{The case of Algorithm P: The descending operation is to be applied on the left and the ascending operation is to be applied on the right.}
\label{Algorithm}
\end{figure}}

It is important to note that unlike the approaches used in our previous work \cite{LDH}, both the descending and ascending operations are used in the same algorithm, depending on the starting point!

As before, we will use $\T^+(D)$ and $\T^-(D)$ to denote the resolving trees obtained by applying Algorithms P and N respectively, and use $\F^+(D)$ and $\F^-(D)$ to denote the set of leaf vertices of $\T^+(D)$ and $\T^-(D)$ respectively.

\section{The main theorem and its proof}\label{s5}

In this section we will prove our main result Theorem \ref{main}. We will begin with some preparations.

Let $\U$ be a leaf vertex in either $\T^+(D)$ or $\T^-(D)$. Since each component is obtained with a fixed starting point, the loop crossings (if they exist) of the component are uniquely determined. Since the components are stacked over each other by the way they are obtained, the sum of the crossing signs between different components is zero. If we smooth the loop crossings, the resulting IS circles are also stacked over each other so the sum of the crossing signs between these different IS circles is also zero. It follows that $w(\U)$ equals the sum of the signs of the loop crossings. Say $\U$ contains $k=\gamma(\U)$ components and the $i$-th component contains $i_j\ge 0$ loop crossings. Smoothing the loop crossings of the $i$-th component results in $i_j+1$ IS circles, so smoothing all loop crossings of $\U$ results in $k+\sum_{1\le j\le k}i_j$ IS circles. By Lemma \ref{new_lemma2}, we have  $k+\sum_{1\le j\le k}i_j\le n$ where $n$ is the number of Seifert circles in $D$. Since $w(\U)$ equals the sum of the signs at the loop crossings of its components, it follows that $E\le n-w(D)-1$ and $e\ge -n-w(D)+1$ by (\ref{apower1}) and (\ref{apower2}), where $E$ and $e$ are maximum and minimum powers of $a$ in $P(D,z,a)$. Furthermore, if $\U\in \F^+(D)$ contains a negative loop crossing, then $w(\U)$ is strictly less than the total number of loop crossings in $\U$ and we will have $\gamma(\U)+w(\U)<n$. Similarly, if $\U\in \F^-(D)$ contains a positive loop crossing, then $\gamma(\U)-w(\U)<n$. We have proven the following lemma.

\begin{lemma}\label{component_path}
Let $E$ and $e$ be the maximum and minimum powers of $a$ in $P(D,z,a)$, then $E\le n-w(D)-1$ and $e\ge -n-w(D)+1$. If a component of $\U\in \F^+(D)$ contains a negative loop crossing, then $\gamma(\U)+w(\U)<n$. Similarly, if a component of $\V\in \F^-(D)$ contains a positive loop crossing, then $\gamma(\U)-w(\U)<n$.
\end{lemma}

For any given component $\Gamma$ of a leaf vertex in either $\T^+(D)$ or $\T^-(D)$, as we travel along it from its starting point (the starting point on the ground floor of the castle), we will eventually exit the castle for the first time through the end point of some floor. The arc of $\Gamma$ from its starting point to this exiting point is called the {\em maximum path} of $\Gamma$. If the ending point $\Gamma$ is the end point on the ground floor, it means that $\Gamma$ is completely contained within the castle and the base Seifert circle. Notice that a maximum path consists of only three kind of line segments: ladders going up or down (these are parts of  crossings), and straight line segments parallel to the ground floor. In the case that the base Seifert circle has clockwise orientation, the ladders going up or down are both from left to right. Furthermore, on either side of a segment that is parallel to the ground floor, there are no crossings (if there were crossings previously there in the original diagram, they are smoothed in the process of obtaining $\Gamma$). This means that if we travel along a strand of the link diagram from a point outside the castle following its orientation, in order to enter a floor below this maximum path, it is necessary for the strand to pass through the path through a ladder from the left in the case of clockwise orientation for the base Seifert circle or from the right in the case of counter clockwise orientation for the base Seifert circle. 

\begin{lemma}\label{lemma_loopfree}
If $\U\in \F^+(D)$ contains a maximum path that does not end on the ground floor,  then the maximum $a$ power in the contribution of $\U$ to $P(D,z,a)$ is smaller than  $n-w(D)-1$. Similarly, if any component of $\V\in \T^-(D)$ contains a maximum path that does not end on the ground floor, then the minimum $a$ power in the contribution of $\V$ to $P(D,z,a)$ is larger than  $-n-w(D)+1$. 
\end{lemma}

Proof. Consider the case of $\U \in \T^+(D)$ first. Assume $\Gamma$ is the first component of $\U$ that contains such a maximum path. Notice that a component that is bounded within the castle and its base Seifert circle contains no loop crossings. So the components before $\Gamma$ are all IS circles of $D$. Say there are $k$ components before $\Gamma$, then by Lemma \ref{lemma4}, there are at most $n-k$ Seifert circles in $D^\p$ where $n$ is the number of Seifert circles in $D$ and $D^\p$ is the link diagram obtained from $D$ after the first $k$ components are removed from it, since Seifert circles in $D^\p$ are IS circles of $D$. Let $C_0$ be the base Seifert circle on which the castle used to derive $\Gamma$ is built.
\begin{figure}[htb!]
\includegraphics[scale=.4]{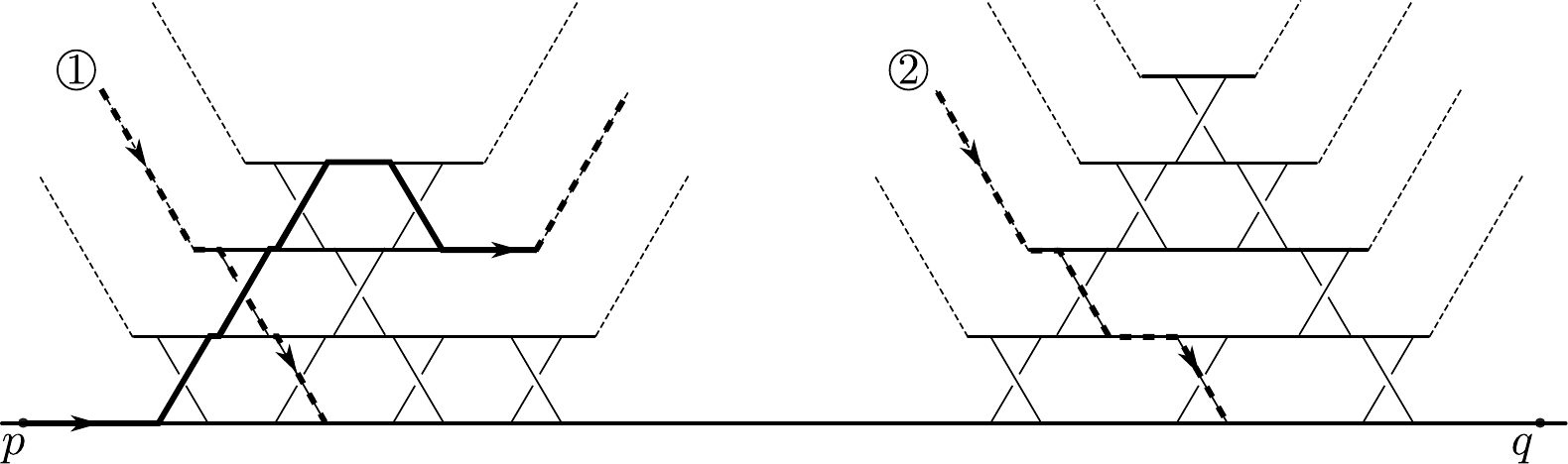}
\caption{A component whose maximum path ends on a floor higher than the ground floor either contains a loop crossing (case 1) or runs through a cycle of Seifert circles (case 2).}
\label{tower}
\end{figure}

Without loss of generality, we assume that $C_0$ has clockwise orientation. Let $F_i$ (of level $i\ge 1$) be the floor where the maximum path of $\Gamma$ exits from its end point and let $T_1$ be the tower that houses the maximum path. If $\Gamma$ is to get back to $C_{0}$ within $T_1$, it will have to cross the maximum path from the left side as shown in Figure \ref{tower} (case 1 marked in the graph), where the maximum path is drawn in double lines, creating a negative loop crossing (since in this case we are applying a descending algorithm). Hence $\gamma(\U)+w(\U)<n$ by Lemma \ref{component_path} and it follows that the $a$-degree of the contribution of $\U$ to $P(D,z,a)$ is less than $n-w(D)-1$ by (\ref{apower1}). If $\Gamma$ does not go back to the base floor within $T_1$, then it has to do so through another tower $T_2$. Let us first smooth the loop crossings of $\Gamma$ that are defined by the starting point on $C_0$. Let $\tilde{C}$ be the resulting IS circle that contains the starting point (which contains the part of $\Gamma$ that is within $T_1$, of course) and let $I_1(D)$ be as defined in Lemma \ref{lemma4}. Let $F_j$ be the highest common floor $T_1$ and $T_2$ share, it is necessary that $j<i$ and it is also necessary that $\Gamma$ returns to $F_j$ from $F_{j+1}^\p$ (the floor above it in $T_2$) within $T_2$. Consider the sub-walk of the Seifert circle walk of $\tilde{C}$, defined by traversing $\tilde{C}$ (following its orientation), starting from its last dividing arc on $C_j$ within $T_1$ and ending at the first dividing arc on $C_j$ within $T_2$. This sub-walk has the form of $C_jC_{j+1}\cdots C^\p_{j+1}C_j$. If this sub-walk contains a sub-walk starting and ending at $C_j$, by the proof of Lemma \ref{lemma4}, it cannot contain a cycle, hence can be reduced to a single Seifert circle $C_j$. However the last step of the reduction would be performed on a sub-walk of the form $C_jC_{j+1}C_j$, which is impossible since the other crossing is either within $T_1$ which is impossible since we started on the last dividing arc on $C_j$ in $T_1$, or the new dividing arc will contain the entire portion of $C_j$ that is outside of $T_1$ which is also impossible since the long IS circle would then not go through $C^\p_{j+1}$ at all. It follows that $C_jC_{j+1}\cdots C^\p_{j+1}C_j$ contains no other closed sub-walk starting and ending at $C_j$, hence it must contain a cycle of $G_S(D)$. By Lemma \ref{lemma4}, it follows that the total number of components in $\U$ plus the total number of loop crossings in them is less than $n$, so $\gamma(\U)+w(\U)<n$ again, and the maximum $a$ power in the contribution of $\U$ to $P(D,z,a)$ is also less than $n-w(D)-1$. The case of $\V\in \T^-(D)$ is similar and is left to our reader. \qed

Lemma \ref{lemma_loopfree} leads to the following two corollaries.

\begin{corollary}\label{Corollary1}
If  the contribution of $\U\in \F^+(D)$ to $P(D,z,a)$ contains an $a^{n-w(D)-1}$ term, then each component of $\U$ is an IS circle of $D$. Similarly, if the contribution of  $\V\in \F^-(D)$  to $P(D,z,a)$ contains an $a^{-n-w(D)+1}$ term, then each component of $\V$ is also an IS circle of $D$.
\end{corollary}

Proof. This is direct from the proof of Lemma  \ref{lemma_loopfree}, since every such component has a maximum path ending at the end point of the ground floor, so the component contains no loop crossings. \qed

The following result provides the proof for the only if part of Theorem \ref{main}. We feel that it is significant enough on its own so we state it as a theorem. Notice that it implies the inequality in the MFW inequality is strict: $a$-span$/2+1<n$.

\begin{theorem}\label{main2}
Let $D$ be a link diagram of a link $L$ such that $G_S(D)$ contains an edge of weight one ($L$ and $D$ need not be alternating) and $n$ be the number of Seifert circles in $D$, then we have $b(L)<n$.
\end{theorem}

\begin{proof}
Let $C^\p$ and $C^{\p\p}$ be two Seifert circles sharing only one crossing between them. If $C^{\p\p}$ shares no crossings with any other Seifert circles, then the crossing between $C^\p$ and $C^{\p\p}$ is nugatory and the statement of the theorem holds. So assume that this is not the case and let $C_1$, $C_2$, ..., $C_j$ be the other Seifert circles sharing crossings with $C^{\p\p}$. The orientations of $C_1$, $C_2$, ..., $C_j$ are the same as that of $C^\p$ and there are no crossings between any two of them. A case of $j=2$ is shown in Figure \ref{circle_reduction}. We will reroute the overpass at the crossing between $C^\p$ and $C^{\p\p}$ along $C_1$, $C_2$, ..., $C_j$ as shown in Figure \ref{circle_reduction} (keeping the strand over the crossings we encounter). 

\begin{figure}[htb!]
\includegraphics[scale=.4]{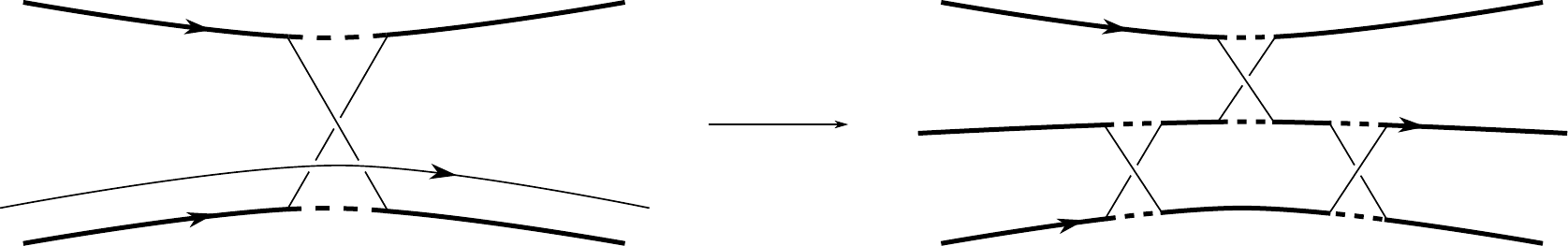}
\caption{The local effect of rerouting the overpass to the Seifert circle structure: the two original Seifert circles remain unchanged.}
\label{local_effect}
\end{figure}

Rerouting the overpass this way will only create new crossings over some crossings between the $C_i$'s and its neighbors other than $C^{\p\p}$. The effect of this rerouted strand to the Seifert circle decomposition structure locally is shown in Figure \ref{local_effect}, which does not change the Seifert circles $C_1$, $C_2$, ..., $C_j$, but the weights of the edges connecting to the vertices corresponding to them in $G_S(D^\p)$ (where $D^\p$ is the new link diagram after the rerouting) may have changed from those of in $G_S(D)$. See Figure \ref{graph_change}.

\begin{figure}[htb!]
\includegraphics[scale=.35]{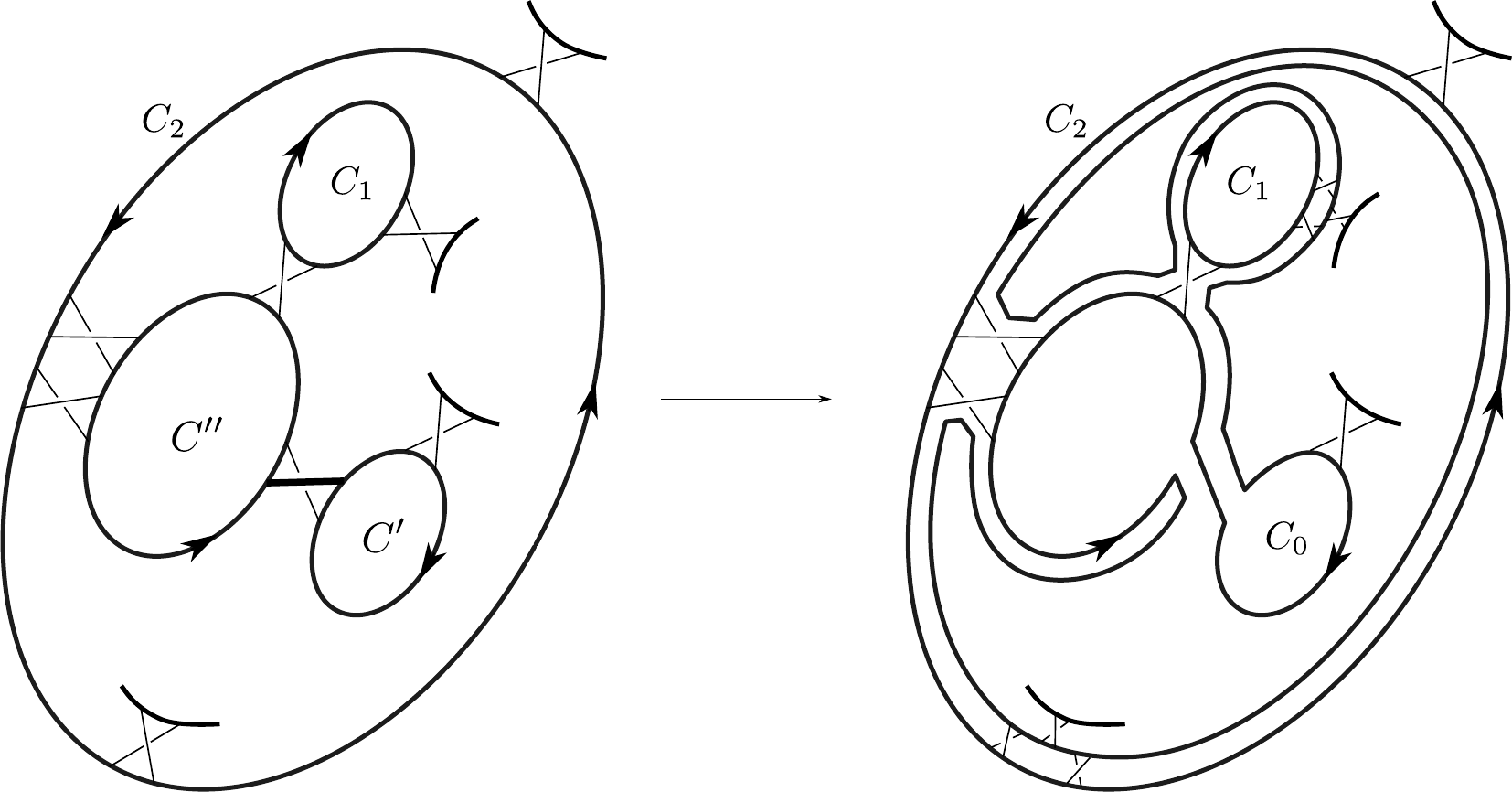}
\caption{Left: The overpass at the single crossing to be rerouted; Right: The two Seifert circles merge into one after the overpass is rerouted as shown.}
\label{circle_reduction}
\end{figure}

\begin{figure}[htb!]
\includegraphics[scale=.7]{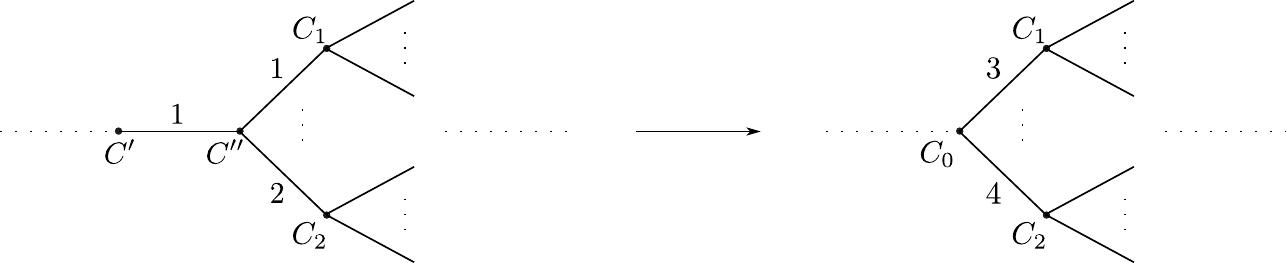}
\caption{How the Seifert graph of the link in Figure \ref{circle_reduction} changed after the rerouting.}
\label{graph_change}
\end{figure}

At the end, we arrive at a new link diagram $D^\p$ that is equivalent to $D$, but with one less Seifert circle. The result then follows since the number of Seifert circles in $D^\p$ is an upper bound of the braid index of $D^\p$, hence $D$.
\end{proof}

If $D$ is a reduced alternating link diagram and $G_S(D)$ contains no edges of weight one, then for any pair of Seifert circles in $D$ that are adjacent in $G_S(D)$, they share at least two crossings and these crossings are of the same signs (in fact all crossings on one side of any Seifert circle of $D$ are of the same sign). Let $\tau^+$ be the total number of positive crossings in $D$, $\sigma^+$ be the number of pairs of Seifert circles in $D$ that share positive crossings, and let $\tau^-$ be the total number of negative crossings in $D$, $\sigma^-$ be the number of pairs of Seifert circles in $D$ that share negative crossings. We are now ready to prove our main theorem.

Proof. ``$\Longrightarrow$" If the braid index of $D$ is $n$ then $G_S(D)$ is free of edges of weight 1 by Theorem \ref{main2}. 

``$\Longleftarrow$" We will prove this in two steps. In the first step, we construct a specific leaf vertex in $\U\in \F^+(D)$ whose contribution to $P(D,z,a)$ contains a term of the form $(-1)^{t^{-}(\U)}z^{t(\U)-n+1}a^{n-w(D)-1}$, where $t^{-}(\U)=\tau^- -2\sigma^-$ and $t(\U)=\tau^++\tau^- -2\sigma^-$. Similarly we construct a specific leaf vertex in $\V\in \F^-(D)$ whose contribution to $P(D,z,a)$ contains a term of the form $(-1)^{t^{-}(\V)+n-1}z^{t(\V)-n+1}a^{-n-w(D)+1}$, where $t^{-}(\V)=\tau^+ -2\sigma^+$ and $t(\V)=\tau^-+\tau^+ -2\sigma^+$. In the second step we show that if a leaf vertex $\U^\p\in \F^+(D)$ makes a contribution to the $a^{n-w(D)-1}$ term in $P(D,z,a)$, then $t^{-}(\U^\p)\le \tau^- -2\sigma^-$. Similarly, if a leaf vertex $\V^\p\in \F^-(D)$ makes a contribution to the $a^{-n-w(D)+1}$ term in $P(D,z,a)$, then $t^{-}(\V^\p)\le \tau^+ -2\sigma^+$. 

Combining the results of the these two steps will then lead to the conclusion of the theorem since the result from the second step implies that $t(\U^\p)\le \tau^++\tau^- -2\sigma^-$ and $t(\V^\p)\le \tau^-+\tau^+ -2\sigma^+$. So $\tau^++\tau^- -2\sigma^-$ is the maximum power of $z$ in the coefficient of the $a^{n-w(D)-1}$ term in $P(D,z,a)$ and $\tau^-+\tau^+ -2\sigma^+$ is the maximum power of $z$ in the coefficient of the $a^{-n-w(D)+1}$ term in $P(D,z,a)$. Furthermore, these maximum powers can only be contributed from $\U^\p\in \F^+(D)$ and $\V^\p\in \F^-(D)$ that are obtained by smoothing all positive crossings of $D$ and all negative crossings of $D$ except two between each pair of Seifert circles sharing negative crossings in the case of $\U^\p$, and by smoothing all negative crossings of $D$ and all positive crossings of $D$ except two between each pair of Seifert circles sharing positive crossings in the case of $\V^\p$. Apparently any such $\U^\p$, $\V^\p$ will make exactly the same contributions to $P(D,z,a)$ as that of $\U$ and $\V$. Thus $E=n-w(D)-1$ and $e=-n-w(D)+1$. So $E-e=2(n-1)$ and $(E-e)/2+1=a$-span$/2+1=n$ and the theorem follows.

Step 1. Choose a castle that is free of trapped Seifert circles. Let $C_0$ be the base Seifert circle of the castle with starting point $p$ and ending point $q$ on its floor. Let us travel along $D$ starting at $p$.

Case 1. The crossings between $C_0$ and its adjacent Seifert circles are all positive. If $C_0$ is clockwise, then we need to apply the descending rule. We will encounter the first crossing from its under strand. We will stay with the component obtained by smoothing this crossing. So we are still traveling on $C_0$ after this crossing is smoothed. We then encounter the next crossing from its under strand and we can again smooth this crossing. Repeating this process, we arrive the first component of $\U$ by smoothing all the crossings between $C_0$ and its adjacent neighbors. If $C_0$ is counter clockwise then we will be applying the ascending rule and we can also obtain a component of $\U$ by smoothing all crossings along $C_0$. It is apparent that after we remove this new component from the diagram, the resulting new diagram is still alternating and has $n-1$ Seifert circles left.

Case 2. The crossings between $C_0$ and its adjacent Seifert circles are all negative. If $C_0$ is clockwise, then we need to apply the descending rule. Let $C_1$ be the first Seifert circle with which $C_0$ shares a crossing as we travel along $C_0$ from $p$. In this case we encounter the first crossing from its over strand. Therefore we have no choice but to keep this crossing. This moves us to $C_1$. Keeping in mind that by the given condition $C_0$ and $C_1$ share at least two crossings and all crossings between $C_1$ and $C_0$ are on the floor of $C_1$ above $F_0$, as we travel on $F_1$ toward the last crossing between $C_1$ and $C_0$, we encounter each crossing from an under strand as one can check. So we can smooth all crossings we encounter (either between $C_1$ and $C_0$ or between $F_1$ and $F_2$ before we reach the last crossing between $C_1$ and $C_0$. We then flip the last crossing to return to $C_0$. If $C_0$ is adjacent to more Seifert circles, we repeat the same procedure. Finally we return to the ending point of $F_0$ and back to the starting point. See Figure \ref{negativecase} for an illustration, where a case of two floors on top of $F_0$ is shown.

\vanish{\begin{figure}[htb!]
\includegraphics[scale=.3]{negativecase}
\caption{How to construct a component starting from a base Seifert circle with negative crossings with its neighbors.}
\label{negativecase}
\end{figure}}

\begin{figure}[htb!]
\includegraphics[scale=.4]{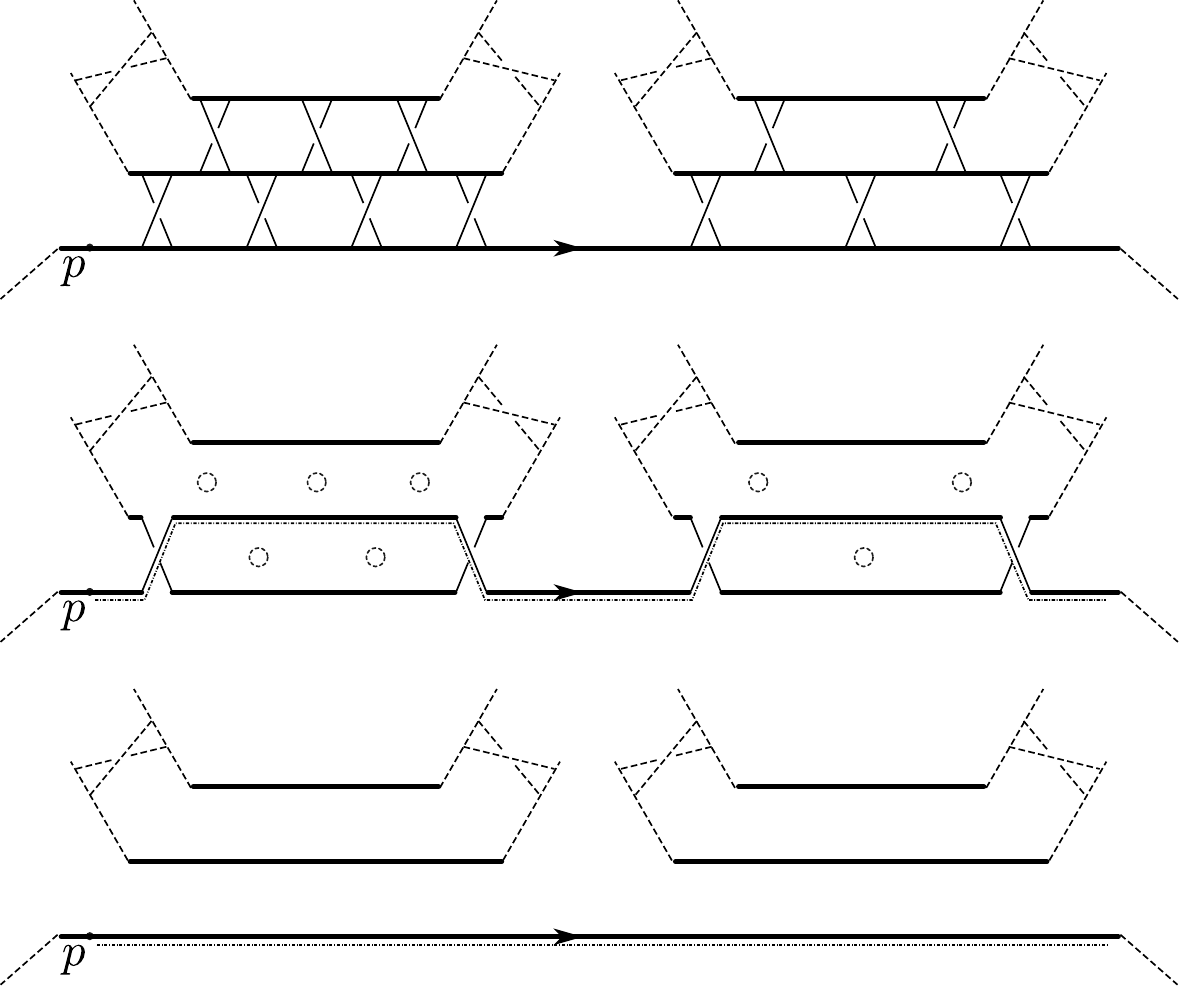}
\caption{How to construct a component starting from a base Seifert circle with negative crossings with its neighbors.}
\label{negativecase}
\end{figure}
 
Since smoothing crossings does not change the alternating nature of a diagram, removing this newly created component will keep the resulting diagram alternating as one can easily see from Figure \ref{negativecase}. In fact, the new diagram is equivalent to the one obtained by smoothing all crossings encountered by the maximum path of the new component (with the new component removed). So it contains $n-1$ Seifert circles. If $C_0$ is counter clockwise, the above argument is the same after we replace the descending algorithm by the ascending algorithm.

So in both cases we created a component and the new diagram is still alternating and contains one less Seifert circle than before. Thus this process can be repeated, at the end we obtain $\U$, which contains $n$ components (each one is an IS circle). And by the way $\U$ is obtained, all positive crossings have been smoothed, and between any pair of Seifert circles that share negative crossings, all but two crossings are smoothed. $\V$ is obtained in a similar manner in which all negative crossings are smoothed, and between any pair of Seifert circles that share positive crossings, all but two crossings are smoothed. This finishes Step 1.

Step 2. Consider a leaf vertex $\U^\p\in \F^+(D)$ that makes a contribution to the $a^{n-w(D)-1}$ term in $P(D,z,a)$. By Lemma \ref{lemma_loopfree}, the maximum path of each component of $\U^\p$ is bounded within its defining castle. Let $\Gamma_1$ be the first component of $\U^\p$. Consider a horizontal segment of the maximum path that represents a local maximum. We leave it to our reader to verify that we will never encounter a negative crossing to the left side of $\Gamma_1$, and all crossings to the right of $\Gamma_1$ are positive and are smoothed. Thus, for a given pair of Seifert circles in $D$ that share negative crossings, if $\Gamma_1$ crosses from one to the other, then at least two crossings between them are not smoothed. If $\Gamma_1$ does not cross from one to the other, then removing $\Gamma_1$ may change parts of these two Seifert circles but will not affect the crossings between these two Seifert circles. The same argument can then be applied to the next component $\Gamma_2$, and so on. It follows that for each pair of Seifert circles sharing negative crossings, at least two crossings cannot be smoothed in $\U^\p$. That is, $\T^{-}(\U^\p)\le \tau^- -2\sigma^-$. Similarly, we have $\T^{-}(\V^\p)\le \tau^+ -2\sigma^+$. \qed

\section{Further discussions}\label{s6}

An immediate consequence of Theorem \ref{main} is that if $D$ is a reduced alternating link diagram and $G_S(D)$ is free of cycles (a tree in the case that link is non-splittable), then the braid index of $D$ is the number of Seifert circles in $D$ since an edge of weight one corresponds to a nugatory crossing in $D$ which do not exist when $D$ is reduced. From the point view of a Seifert graph, one can compare the previously known results to our result. For example,
in \cite{Mu}, Murasugi showed that if a reduced alternating link diagram is a star product of elementary torus links (namely torus links of the form $(2,k)$ with $k\ge 2$), then the braid index of the link equals the number of Seifert circles in the diagram. Links obtained this include reduced alternating closed braids, for which we gave an alternative proof in \cite{LDH}. The Seifert graph of such a link diagram is a single path. 

For two Seifert circles of a link diagram $D$ sharing one single crossing, call them a {\em mergeable pair} since the operation used in the proof of Theorem \ref{main2} merges them into one Seifert circle. Two mergeable pairs are said to have distance $k$ if the shortest path in $G_S(D)$ from any vertex corresponding to a Seifert circle in one mergeable pair to that a vertex corresponding to a Seifert circle in the other mergeable pair contains $k$ edges. Since the operation in the proof of Theorem \ref{main2} used to combine a mergeable pair of Seifert circles does not affect the mergeable pairs that are of a distance $k\ge 2$ away, this operation can be applied to these ``far away mergeable pairs", which gives us the following more general result about a link diagram that is not necessarily alternating.

\begin{theorem}\label{main3}
If a link diagram $D$ contains $m$ mergeable pairs such that each pair is of a distance at least 2 from the rest, then the braid index of $D$ is at most $n-m$, where $n$ is the number of Seifert circles of $D$.
\end{theorem}

\begin{corollary}\label{Co_main3}
If a link diagram $D$ contains $m$ mergeable pairs such that each pair is of a distance at least 2 from the rest, then $a$-span$/2+1\le n-m$ for the $a$-span of $P(D,z,a)$ where $n$ is the number of Seifert circles of $D$. In particular, if the equality holds, then the braid index of $D$ is equal to $n-m$.
\end{corollary}

Of course, a natural next step would be identify those alternating links whose reduced alternating diagrams contain edges of weight one in their Seifert graphs, for which the equality in the above generalized MFW inequality holds. 

\smallskip
\section*{Acknowledgement}
The research of the second author was partially supported by a grant from
the Simons Foundation (\#245153 to G\'abor Hetyei).

\end{document}